\documentclass[12pt]{amsart}
\usepackage{amsmath,amssymb,amsthm,amscd,amsfonts,geometry,color}
\renewcommand{\labelenumi}{\rm(\theenumi)}

\theoremstyle{plain}
\newtheorem{thm}{Theorem}[section]
\newtheorem{prop}[thm]{Proposition}
\newtheorem{cor}[thm]{Corollary}
\newtheorem{lem}[thm]{Lemma}

\newtheorem{main}{Main Theorem}

\theoremstyle{remark}
\newtheorem{remark}{Remark}

\theoremstyle{definition}

\newcommand{\I}{\mathbf{I}}                       % unit interval
\newcommand{\0}{\mathbf{0}}                       % zero vector
\newcommand{\e}{\mathbf{e}}                       % unit vector

\newcommand{\cl}{\operatorname{cl}}               % closure
\newcommand{\intr}{\operatorname{int}}            % interior
\newcommand{\pr}{\operatorname{pr}}               % projection

\newcommand{\card}{\operatorname{card}}	          % cardinality
\begin{document}

\title[Manifolds modeled on absorbing sets and the discrete cells property]{Characterizations of topological manifolds modeled on absorbing sets in non-separable Hilbert spaces and the discrete cells property}
\author{Katsuhisa Koshino}
\address[Katsuhisa Koshino]{Faculty of Engineering, Kanagawa University, Yokohama, 221-8686, Japan}
\email{ft160229no@kanagawa-u.ac.jp}
\maketitle

\begin{abstract}
In this paper, we characterize infinite-dimensional manifolds modeled on absorbing sets in non-separable Hilbert spaces by using the discrete cells property,
 which is a general position property.
Moreover, we study the discrete (locally finite) approximation property,
 which is an extension of the discrete cells property.
\end{abstract}

\section{Introduction}

Throughout this paper, spaces are metrizable, maps are continuous, and $\kappa$ is an infinite cardinal.
Let $\ell_2(\kappa)$ be the Hilbert space of density $\kappa$, and $\ell_2^f(\kappa)$ be the linear subspace spanned by the canonical orthonormal basis of $\ell_2(\kappa)$,
 that is,
 $$\ell_2^f(\kappa) = \{(x(\gamma))_{\gamma < \kappa} \in \ell_2(\kappa) \mid x(\gamma) = 0 \text{ except for finitely many } \gamma < \kappa\}.$$
Given a space $E$, we call a space $X$ to be an \textit{$E$-manifold} if each point $x \in X$ has an open neighborhood homeomorphic to some open subset of $E$.
For a class $\mathcal{C}$ of spaces and a cardinal $\lambda$, write
\begin{itemize}
 \item $\mathfrak{C}_\sigma = \{\bigcup_{n \in \omega} A_n \mid A_n \in \mathfrak{C} \text{ and } A_n \text{ is closed}\}$,
 \item $\bigoplus_\lambda \mathfrak{C} = \{\bigoplus_{\gamma < \lambda} A_\gamma \mid A_\gamma \in \mathfrak{C}\}$,
 \item $\mathfrak{C}(\lambda) = \{A \in \mathfrak{C} \mid A \text{ is of density} \leq \lambda\}$.
\end{itemize}

A closed set $A$ in a space $X$ is called a \textit{$Z$-set} in $X$ if for each open cover $\mathcal{U}$ of $X$, there exists a map $f : X \to X$ such that $f$ is $\mathcal{U}$-close to the identity map of $X$ and $f(X) \cap A = \emptyset$.
When the closure $\cl{f(X)}$ misses $A$,
 it is said to be a \textit{strong $Z$-set} in $X$.
Recall that for maps $f : X \to Y$ and $g : X \to Y$, and for an open cover $\mathcal{U}$ of $Y$, $f$ is \textit{$\mathcal{U}$-close} to $g$ provided that for any point $x \in X$, there is a member $U \in \mathcal{U}$ such that $\{f(x), g(x)\} \subset U$.
A \textit{(strong) $Z_\sigma$-set} is a countable union of (strong) $Z$-sets.
A map $f : X \to Y$ is a \textit{$Z$-embedding} if $f$ is an embedding and the image $f(X)$ is a $Z$-set in $Y$.
Given a class $\mathfrak{C}$, we say that a space $X$ is \textit{strongly $\mathfrak{C}$-universal} or \textit{strongly universal for $\mathfrak{C}$} if the following condition is satisfied.
\begin{itemize}
 \item Let $f : A \to X$ be a map from a space in $\mathfrak{C}$.
 Suppose that $B$ is a closed subset of $A$ and the restriction $f|_B$ is a $Z$-embedding.
 Then for each open cover $\mathcal{U}$ of $X$, there exists a $Z$-embedding $g : A \to X$ such that $g$ is $\mathcal{U}$-close to $f$ and $g|_B = f|_B$.
\end{itemize}
For spaces $X \subset M$, $X$ is said to be a \textit{$\mathfrak{C}$-absorbing set} or an \textit{absorbing set for $\mathfrak{C}$} in $M$,
 which plays an important role in the theory of infinite-dimensional topology, provided that it satisfies the following conditions.
\begin{enumerate}
 \item $X$ is homotopy dense in $M$ and $X \in \mathfrak{C}_\sigma$.
 \item $X$ is strongly $\mathfrak{C}$-universal.
 \item $X$ is a strong $Z_\sigma$-set in itself.
\end{enumerate}
Here $X$ is \textit{homotopy dense} in $M$ if there exists a homotopy $h : M \times \I \to M$ such that $h(M \times (0,1]) \subset X$ and $h(x,0) = x$ for each $x \in M$.
For a cardinal $\lambda$, a space $X$ has \textit{the $\lambda$-discrete (locally finite) approximation property for $\mathfrak{C}$} provided that the following condition holds.
\begin{itemize}
 \item Let $f : \bigoplus_{\gamma < \lambda} A_\gamma \to X$ be a map from a topological sum of members of $\mathfrak{C}$.
 For every open cover $\mathcal{U}$ of $X$, there is a map $g : \bigoplus_{\gamma < \lambda} A_\gamma \to X$ such that $g$ is $\mathcal{U}$-close to $f$ and the family $\{g(A_\gamma) \mid \gamma < \lambda\}$ is discrete (locally finite) in $X$.
\end{itemize}
In particularly, when $\mathfrak{C} = \{\I^n\}$, $n \in \omega$,
 we say that $X$ has \textit{the $\lambda$-discrete (locally finite) $n$-cells property}.
These properties are general position properties that manifolds have based on ``dimension'' in the finite-dimensional case or ``density'' in the infinite-dimensional case.
They are very useful for recognizing manifolds.
For example, H.~Toru\'nczyk \cite{Tor5,Tor6} characterized $\ell_2(\kappa)$-manifolds by using the $\kappa$-discrete $n$-cells property.
We will study these properties in Section~\ref{D(LF)AP}.

It is said that a class $\mathfrak{C}$ is \textit{topological} if every space homeomorphic to some member of $\mathfrak{C}$ also belongs to $\mathfrak{C}$,
 and $\mathfrak{C}$ is \textit{closed hereditary} if any closed subspace of some member of $\mathfrak{C}$ also belongs to $\mathfrak{C}$.
Using the discrete cells property, we shall characterize infinite-dimensional manifolds modeled on absorbing sets in Hilbert spaces as follows,
 which generalize and improve the previous results in \cite{BeMo,SaY,Min,Sa11}.

\begin{thm}\label{abs.char.}
Let $\mathfrak{C}$ be a topological and closed hereditary class, and $\Omega$ be a $\mathfrak{C}$-absorbing set in $\ell_2(\kappa)$.
For a connected space $X \in \mathfrak{C}_\sigma$ of density $\leq \kappa$, the following are equivalent.
\begin{enumerate}
 \item $X$ is an $\Omega$-manifold.
 \item $X$ can be embedded into some $\ell_2(\kappa)$-manifold as a $\mathfrak{C}$-absorbing set.
 \item $X$ satisfies the following conditions:
 \begin{enumerate}
  \item $X$ is an ANR;
  \item $X$ is strongly $\mathfrak{C}$-universal;
  \item $X$ has the $\kappa$-discrete $n$-cells property for every $n \in \omega$;
  \item $X$ is a strong $Z_\sigma$-set in itself.
 \end{enumerate}
\end{enumerate}
\end{thm}

A class $\mathfrak{C}$ is \textit{$\I$-stable} provided that for each $A \in \mathfrak{C}$, the product $A \times \I \in \mathfrak{C}$.
As is easily observed,
 the class $\mathfrak{M}_0^{fd}$ of finite-dimensional compact spaces is topological, closed hereditary and $\I$-stable.
The space $\ell_2^f(\kappa)$ is a $\bigoplus_\kappa \mathfrak{M}_0^{fd}$-absorbing set in $\ell_2(\kappa)$.
In Section~\ref{abs.}, we will prove the following:

\begin{thm}\label{abs.discr.}
Let $\aleph_0 \leq \kappa' < \kappa$ and $\mathfrak{C}$ be a topological, closed hereditary and $\I$-stable class.
If $\Omega$ is a $\mathfrak{C}(\kappa')$-absorbing set in $\ell_2(\kappa')$,
 then $\ell_2^f(\kappa) \times \Omega$ is a $\bigoplus_{\kappa} \mathfrak{C}(\kappa')$-absorbing set in $\ell_2(\kappa) \times \ell_2(\kappa')$.
\end{thm}

By the discrete cells property, we shall establish the following characterization.

\begin{main}
Let $\aleph_0 \leq \kappa' < \kappa$.
Suppose that $\mathfrak{C}$ is a topological, closed hereditary and $\I$-stable class,
 and that $\Omega$ is a $\mathfrak{C}(\kappa')$-absorbing set in $\ell_2(\kappa')$.
A connected space $X \in (\bigoplus_{\kappa} \mathfrak{C}(\kappa'))_\sigma$ is an $(\ell_2^f(\kappa) \times \Omega)$-manifold if and only if the following conditions are satisfied.
\begin{enumerate}
 \item $X$ is an ANR.
 \item $X$ is strongly $\mathfrak{C}(\kappa')$-universal.
 \item $X$ has the $\kappa$-discrete $n$-cells property for any $n \in \omega$.
 \item $X$ is a strong $Z_\sigma$-set in itself.
\end{enumerate}
\end{main}

\section{Characterizations of infinite-dimensional manifolds}

In this section, some characterizations of infinite-dimensional manifolds related on Main Theorem and its applications are introduced.
A space $X$ is called to be \textit{$\sigma$-locally compact} if it is a countable union of locally compact subspaces.
Note that any $\sigma$-locally compact space can be expressed as a countable union of closed subspaces that are discrete unions of compact sets, refer to \cite[Proposition~5.1 and Remark~3]{Kos1}.
In \cite{SaY} (cf.~\cite{Mog}), K.~Sakai and M.~Yaguchi characterized $\ell_2^f(\kappa)$-manifolds as follows:

\begin{thm}
A connected space $X$ is an $\ell_2^f(\kappa)$-manifold if and only if the following conditions hold.
\begin{enumerate}
 \item $X$ is a strongly countable-dimensional, $\sigma$-locally compact ANR of density $\leq \kappa$.
 \item $X$ is strongly universal for the class of strongly countable-dimensional, locally compact spaces of density $\leq \kappa$.
 \item $X$ is a strong $Z_\sigma$-set in itself.
\end{enumerate}
\end{thm}

The theory of infinite-dimensional topology is applied for the study of function spaces and hyperspaces because they are often homeomorphic to typical infinite-dimensional spaces.
For example, some function spaces consisting of PL maps and some hyperspaces consisting of finite subsets can be homeomorphic to $\ell_2^f(\kappa)$,
 see \cite{Sakaik8,CN,Kos8}.
When we detect $\ell_2^f(\kappa)$-manifolds among spaces by using the above characterization,
 it will be difficult to verify the strong universality.
As is easily observed,
 the strong universality for the class of strongly countable-dimensional, locally compact spaces of density $\leq \kappa$ implies the one for $\mathfrak{M}_0^{fd}$ and the $\kappa$-discrete $n$-cells property for every $n \in \omega$.
By virtue of the discrete cells property, K.~Sakai and M.~Yaguchi's result was improved as follows \cite{Kos1}:

\begin{thm}\label{l2f}
A connected space $X$ is an $\ell_2^f(\kappa)$-manifold if and only if the following conditions are satisfied.
\begin{enumerate}
 \item $X$ is a strongly countable-dimensional, $\sigma$-locally compact ANR of density $\leq \kappa$.
 \item $X$ is strongly $\mathfrak{M}_0^{fd}$-universal.
 \item $X$ has the $\kappa$-discrete $n$-cells property for each $n \in \omega$.
 \item Every finite-dimensional compact subset of $X$ is a strong $Z$-set\footnote{The both conditions (4) of Theorems~\ref{l2f} and \ref{f_sigma} can be replaced with the one that $X$ is a strong $Z_\sigma$-set in itself.}.
\end{enumerate}
\end{thm}

Let $\alpha$ be a countable ordinal.
Given a space $X$, let $\mathfrak{A}_0(X)$ be the collection of open sets in $X$, and let $\mathfrak{M}_0(X)$ be the one of closed sets in $X$.
By transfinite induction, for $\alpha \geq 1$, we define the collections $\mathfrak{A}_\alpha(X)$ and $\mathfrak{M}_\alpha(X)$ as follows:
 $$\mathfrak{A}_\alpha(X) = \Bigg\{\bigcup_{n \in \omega} A_n \Bigg| A_n \in \bigcup_{\gamma < \alpha} \mathfrak{M}_\gamma(X)\Bigg\} \text{ and } \mathfrak{M}_\alpha(X) = \Bigg\{\bigcap_{n \in \omega} A_n \Bigg| A_n \in \bigcup_{\gamma < \alpha} \mathfrak{A}_\gamma(X)\Bigg\}.$$
Let $\mathfrak{A}_\alpha$ be the class such that $X \in \mathfrak{A}_\alpha$ if $X \in \mathfrak{A}_\alpha(Y)$ for any space $Y$ containing $X$ as a subspace, and let $\mathfrak{M}_\alpha$ be the one such that $X \in \mathfrak{M}_\alpha$ if $X \in \mathfrak{M}_\alpha(Y)$ for any space $Y$ containing $X$ as a subspace.
We call $\mathfrak{A}_\alpha$ and $\mathfrak{M}_\alpha$ \textit{the absolute Borel classes}.
The absolute Borel classes are topological, closed hereditary and $\I$-stable.
Remark that $\mathfrak{A}_0 = \emptyset$.
Moreover, $\mathfrak{M}_0$ is the class of compact spaces,
 $\mathfrak{M}_1$ is the one of completely metrizable spaces,
 $\mathfrak{A}_1$ is the one of $\sigma$-locally compact spaces.
It is known that for $\alpha \geq 1$, there exist the absorbing sets $A_\alpha(\kappa)$ and $M_\alpha(\kappa)$ for the absolute Borel classes $\mathfrak{A}_\alpha(\kappa)$ and $\mathfrak{M}_\alpha(\kappa)$ in the Hilbert space $\ell_2(\kappa)$ respectively,
 see \cite{BeMo,SaY,Min}.
The spaces $\ell_2(\kappa) \times \ell_2^f(\aleph_0)$ and $\ell_2^f(\kappa) \times \I^\omega$, where $\I^\omega$ is the Hilbert cube,
 can be regard as absorbing sets for $\mathfrak{M}_1(\kappa)$ and $\mathfrak{A}_1(\kappa)$ in $\ell_2(\kappa)$ respectively.

A space $X$ is \textit{$\sigma$-completely metrizable} if it can be written as a countable union of completely metrizable, closed subspaces.
It is said that $X$ is \textit{locally density $\leq \kappa$} if every $x \in X$ has a neighborhood of density $\leq \kappa$.
A space is \textit{$\sigma$-locally density $\leq \kappa$} if it can be expressed as a countable union of closed subspaces that are locally density $\leq \kappa$.
Especially, when $\kappa = \aleph_0$,
 it is called to be \textit{$\sigma$-locally separable}.
Infinite-dimensional manifolds modeled on the product spaces $\ell_2(\kappa) \times \ell_2^f(\aleph_0)$, $\ell_2^f(\kappa) \times \ell_2(\aleph_0)$ and $\ell_2^f(\kappa) \times \I^\omega$ are characterized as follows,
 refer to \cite{Mog,SaY,Kos1,Sa11}\footnote{In \cite{Sa11}, their characterizations by the discrete cells property are given.}.

\begin{thm}
A connected space $X$ is an $(\ell_2(\kappa) \times \ell_2^f(\aleph_0))$-manifold if and only if the following conditions are satisfied.
\begin{enumerate}
 \item $X$ is a $\sigma$-completely metrizable ANR of density $\leq \kappa$.
 \item $X$ is strongly $\mathfrak{M}_1(\kappa)$-universal.
 \item $X$ has the $\kappa$-discrete $n$-cells property for any $n \in \omega$.
 \item $X$ is a strong $Z_\sigma$-set in itself.
\end{enumerate}
\end{thm}

\begin{thm}\label{g_delta}
A connected space $X$ is an $(\ell_2^f(\kappa) \times \ell_2(\aleph_0))$-manifold if and only if the following conditions are satisfied.
\begin{enumerate}
 \item $X$ is a $\sigma$-locally separable, $\sigma$-completely metrizable ANR of density $\leq \kappa$.
 \item $X$ is strongly $\mathfrak{M}_1(\aleph_0)$-universal.
 \item $X$ has the $\kappa$-discrete $n$-cells property for any $n \in \omega$.
 \item $X$ is a strong $Z_\sigma$-set in itself.
\end{enumerate}
\end{thm}

\begin{thm}\label{f_sigma}
A connected space $X$ is an $(\ell_2^f(\kappa) \times \I^\omega)$-manifold if and only if the following conditions are satisfied.
\begin{enumerate}
 \item $X$ is a $\sigma$-locally compact ANR of density $\leq \kappa$.
 \item $X$ is strongly $\mathfrak{M}_0$-universal.
 \item $X$ has the $\kappa$-discrete $n$-cells property for any $n \in \omega$.
 \item Each compact subset of $X$ is a strong $Z$-set\footnotemark[1].
\end{enumerate}
\end{thm}

\begin{remark}
In the case that $k > \aleph_0$, Theorems~\ref{g_delta} and \ref{f_sigma} can be obtained as corollaries of Main Theorem.
Indeed, the space $\ell_2^f(\kappa) \times \ell_2(\aleph_0)$ is homeomorphic to $\ell_2^f(\kappa) \times \ell_2^f(\aleph_0) \times \ell_2(\aleph_0)$,
 and $\ell_2^f(\aleph_0) \times \ell_2(\aleph_0)$ is homeomorphic to an $\mathfrak{M}_1(\aleph_0)$-absorbing set in $\ell_2(\aleph_0)$.
Moreover, the space $\ell_2^f(\kappa) \times \I^\omega$ is homeomorphic to $\ell_2^f(\kappa) \times \ell_2^f(\aleph_0) \times \I^\omega$,
 and $\ell_2^f(\aleph_0) \times \I^\omega$ is homeomorphic to an $\mathfrak{M}_0$-absorbing set in $\ell_2(\aleph_0)$.
\end{remark}

Let denote $\mathfrak{A}_\alpha$ and $\mathfrak{M}_\alpha$ by $\mathfrak{B}_\alpha$, and $A_\alpha(\kappa)$ and $M_\alpha(\kappa)$ by $\Omega_\alpha(\kappa)$ for simplicity.
For $\alpha \geq 2$, M.~Bestvina, J.~Mogilski, and K.~Mine \cite{BeMo,Min} gave characterizations to manifolds modeled on the absorbing sets $\Omega_\alpha(\kappa)$ for the absolute Borel classes $\mathfrak{B}_\alpha(\kappa)$ in the Hilbert space $\ell_2(\kappa)$ as follows:

\begin{thm}
For a countable ordinal $\alpha \geq 2$, a connected space $X$ is an $\Omega_\alpha(\kappa)$-manifold if and only if the following conditions hold.
\begin{enumerate}
 \item $X$ is an ANR and $X \in \mathfrak{B}_\alpha(\kappa)$.
 \item $X$ is strongly $\mathfrak{B}_\alpha(\kappa)$-universal.
 \item $X$ is a strong $Z_\sigma$-set in itself.
\end{enumerate}
\end{thm}

As a corollary of Main Theorem, we have the following:

\begin{cor}\label{sep.}
Let $\aleph_0 \leq \kappa' < \kappa$ and $\alpha \geq 2$ be a countable ordinal.
A connected space $X$ is an $(\ell_2^f(\kappa) \times \Omega_\alpha(\kappa'))$-manifold if and only if the following conditions are satisfied.
\begin{enumerate}
 \item $X$ is a $\sigma$-locally density $\leq \kappa'$ ANR and $X \in \mathfrak{B}_\alpha(\kappa)$.
 \item $X$ is strongly $\mathfrak{B}_\alpha(\kappa')$-universal.
 \item $X$ has the $\kappa$-discrete $n$-cells property for every $n \in \omega$.
 \item $X$ is a strong $Z_\sigma$-set in itself.
\end{enumerate}
\end{cor}

\begin{remark}
Consider the case that $\kappa = \kappa' \geq \aleph_0$.
If $\mathfrak{C}(\kappa) = \bigoplus_\kappa \mathfrak{C}(\kappa)$,
 then $\ell_2^f(\kappa) \times \Omega$ is homeomorphic to $\Omega$.
Especially, when $\mathfrak{C} = \mathfrak{B}_\alpha$, $\alpha \geq 2$,
 the space $\ell_2^f(\kappa) \times \Omega_\alpha(\kappa)$ is homeomorphic to $\Omega_\alpha(\kappa)$.
\end{remark}

\section{$Z$-set and the strong universality}

In this section, we list some results on $Z$-set and the strong universality.
The following properties of (strong) $Z$-sets in ANRs are useful, see Corollary~1.4.5 and Exercise~1.4.1 of \cite{BRZ}.

\begin{prop}\label{Z-op.union}
Let $X$ be an ANR.
\begin{enumerate}
 \item For every (strong) $Z$-set $A$ in $X$ and every open subset $U$ of $X$, $A \cap U$ is a (strong) $Z$-set in $U$.
 \item A locally finite union of (strong) $Z$-sets in $X$ is a (strong) $Z$-set.
\end{enumerate}
\end{prop}

Due to \cite[Lemma~3.3 and Proposition~3.4]{SaY} (cf.~Corollary~1.8 in \cite{BeMo}) and \cite[Proposition~3.11]{SaY}, we have the following proposition.

\begin{prop}\label{Z-DCP}
Suppose that $X$ is an ANR with the $\kappa$-discrete $n$-cells property for all $n \in \omega$.
\begin{enumerate}
 \item If a closed subset $A \subset X$ is compact or of density $< \kappa$,
 then $A$ is a $Z$-set in $X$.
 \item If $A$ is a $Z$-set and a strong $Z_\sigma$-set in $X$,
 then $A$ is a strong $Z$-set.
\end{enumerate}
\end{prop}

Using Theorem~2.6.5 of \cite{Sa10}, we can show the following:

\begin{prop}\label{str.Z_sigma-g-hered.}
For any ANR $X$, if each point $x \in X$ has an open neighborhood that is a strong $Z_\sigma$-set in itself,
 then $X$ is also a strong $Z_\sigma$-set in itself.
\end{prop}

A map $f : X \to Y$ is \textit{closed over} $A \subset Y$ provided that for each point $a \in A$ and each open neighborhood $U$ of $f^{-1}(a)$ in $X$, there is an open neighborhood $V$ of $a$ in $Y$ such that $f^{-1}(V) \subset U$.
A (strong) $Z$-set is characterized by the homotopy density, see Corollary~1.2 of \cite{BeMo}.

\begin{lem}\label{Z-homot.dense}
Let $X$ be an ANR and $A$ be a closed subset of $X$.
Then the following are equivalent.
\begin{enumerate}
 \item $A$ is a (strong) $Z$-set in $X$.
 \item There is a homotopy $h : X \times \I \to X$ such that $h(X \times (0,1]) \subset X \setminus A$, $h(x,0) = x$ for each $x \in X$ (and $h$ is closed over $A$).
\end{enumerate}
\end{lem}

On (strong) $Z$-sets in product spaces, we prove the following proposition.

\begin{prop}\label{Z-prod.}
Let $X$ and $Y$ be ANRs.
For every (strong) $Z$-set $A$ in $X$, $A \times Y$ is a (strong) $Z$-set in $X \times Y$.
\end{prop}

\begin{proof}
We will show the case that $A$ is a strong $Z$-set in $X$.
According to Lemma~\ref{Z-homot.dense}, there exists a homotopy $h : X \times \I \to X$ such that $h(X \times (0,1]) \subset X \setminus A$, $h(x,0) = x$ for any $x \in X$, and $h$ is closed over $A$.
We define the desired homotopy $\tilde{h} : X \times Y \times \I \to X \times Y$ by $\tilde{h}(x,y,t) = (h(x,t),y)$.
It is clear that
 $$\tilde{h}(X \times Y \times (0,1]) \subset X \times Y \setminus A \times Y$$
 and $\tilde{h}(x,y,0) = (x,y)$ for all $(x,y) \in X \times Y$.
It remains to prove that $\tilde{h}$ is closed over $A \times Y$.
Take any $(a,y) \in A \times Y$ and any open neighborhood $U$ of $\tilde{h}^{-1}(a,y)$ in $X \times Y \times \I$.
Remark that $\tilde{h}^{-1}(a,y) = \{(a,y,0)\}$ and $h^{-1}(a) = \{(a,0)\}$.
Hence we can choose an open neighborhood $U_1$ of $(a,0)$ in $X \times \I$ and an open neighborhood $U_2$ of $y$ in $Y$ such that $U_1 \times U_2 \subset U$.
Since $h$ is closed over $A$,
 there is an open neighborhood $V$ of $a$ in $X$ such that $h^{-1}(V) \subset U_1$.
Then $V \times U_2$ is an open neighborhood of $(a,y)$ in $X \times Y$ and
 $$\tilde{h}^{-1}(V \times U_2) \subset h^{-1}(V) \times U_2 \subset U_1 \times U_2 \subset U.$$
The proof is complete.
\end{proof}

Proposition~2.1 of \cite{BeMo} is valid without separability.

\begin{prop}\label{str.univ.op.}
Let $\mathfrak{C}$ be a topological and closed hereditary class.
Suppose that an ANR $X$ is strongly $\mathfrak{C}$-universal.
Then so is every open set in $X$.
\end{prop}

Applying Theorem~2.6.5 of \cite{Sa10}, we can prove the following proposition, refer to \cite[Proposition~2.7]{BeMo} and \cite{SaY}.

\begin{prop}\label{str.univ.g-hered.}
Let $\mathfrak{C}$ be a topological and closed hereditary class of spaces and $X$ be an ANR.
If every point $x \in X$ has an open neighborhood that is strongly $\mathfrak{C}$-universal,
 then $X$ is also strongly $\mathfrak{C}$-universal.
\end{prop}

Proposition~2.6 of \cite{BeMo} holds without separability.

\begin{prop}\label{str.univ.prod.}
Let $\mathfrak{C}$ be a topological and $\I$-stable class.
Suppose that $X$ and $Y$ are ANRs such that every $Z$-set in $X \times Y$ is a strong $Z$-set.
If $X$ is strongly $\mathfrak{C}$-universal,
 then so is the product space $X \times Y$.
\end{prop}

According to Proposition~3.5 of \cite{SaY}, we have the following\footnote{A class $\mathfrak{C}$ is \textit{additive} if every space written as a union of two closed subsets belonging to $\mathfrak{C}$ is in $\mathfrak{C}$.
Proposition~3.5 of \cite{SaY} is valid without additivity.
Moreover, this proposition can be shown by using the $\kappa$-discrete $n$-cells property instead of the assumption that $\I^n \times \kappa \in \mathfrak{C}$.}:

\begin{prop}\label{str.univ.sigma}
Let $\mathfrak{C}$ be a topological and closed hereditary class of spaces and $X$ be an ANR of density $\leq \kappa$.
Suppose that $X$ has the $\kappa$-discrete $n$-cells property for each $n \in \omega$ and is a strong $Z_\sigma$-set in itself.
If $X$ is strongly $\mathfrak{C}$-universal,
 then it is strongly $\mathfrak{C}_\sigma$-universal.
\end{prop}

\section{The discrete approximation property}\label{D(LF)AP}

This section is devoted to investigating the discrete (locally finite) approximation property.
For a homotopy $h : X \times \I \to Y$ and an open cover $\mathcal{U}$ of $Y$, $h$ is a \textit{$\mathcal{U}$-homotopy} provided that for each $x \in X$, the homotopy track $h(\{x\} \times \I)$ is contained in some member of $\mathcal{U}$.
When $Y = (Y,d)$ is a metric space,
 for a positive number $\epsilon > 0$, $h$ is an \textit{$\epsilon$-homotopy} if the diameter of $h(\{x\} \times \I)$ is less than $\epsilon$.
The locally finite approximation property is open and homotopy dense hereditary.

\begin{prop}\label{LFAP-op.}
Let $\mathfrak{C}$ be a class.
Suppose that $X$ is an ANR and $U$ is an open set in $X$.
If $X$ has the $\kappa$-locally finite approximation property for $\mathfrak{C}$,
 then so does $U$.
\end{prop}

\begin{proof}
Suppose that $f : A = \bigoplus_{\gamma < \kappa} A_\gamma \to U$ is a map,
 where $A_\gamma \in \mathfrak{C}$ for all $\gamma < \kappa$,
 and $\mathcal{U}$ is an open cover of $U$.
We will construct a map $h : A \to U$ such that $h$ is $\mathcal{U}$-close to $f$ and the family $\{h(A_\gamma) \mid \gamma < \kappa\}$ is locally finite in $U$.

Take an open cover $\mathcal{U}'$ of $U$ such that it is a star-refinement of $\mathcal{U}$.
We can write $U = \bigcup_{i \in \omega} U_i$,
 where $U_i$ is an open set in $X$ and $\cl{U_i} \subset U_{i+1}$ for every $i \in \omega$.
Let
 $$B_i = f^{-1}(\cl{U_i} \setminus U_{i-1}),$$
 $i \in \omega$,
 where $U_{-1} = \emptyset$.
We define an open cover $\mathcal{U}''$ of $U$ and open covers $\mathcal{V}_i$ of $X$, $i \in \omega$, as follows:
 $$\mathcal{U}'' = \{U' \cap U_i \setminus \cl{U_{i-2}} \mid U' \in \mathcal{U}', i > 0\} \text{ and } \mathcal{V}_i = \{U'' \cap U_{2i + 1} \mid U'' \in \mathcal{U}''\} \cup \{X \setminus \cl{U_{2i}}\}.$$
Since $X$ has the $\kappa$-locally finite approximation property for $\mathfrak{C}$,
 there exist maps $g_i : A \to X$, $i \in \omega$, such that $g_i$ is $\mathcal{V}_i$-homotopic to $f$ and $\{g_i(A_\gamma) \mid \gamma < \kappa\}$ is locally finite in $X$.
Then $g_i|_{B_{2i}}$ is $\mathcal{U}''$-homotopic to $f|_{B_{2i}}$ for all $i \in \omega$.
By the Homotopy Extension Theorem, we can obtain a map $g : A \to U$ such that $g$ is $\mathcal{U}''$-homotopic to $f$ and $g|_{B_{2i}} = g_i|_{B_{2i}}$ for each $i \in \omega$.
It is easy to see that $\{g(A_\gamma \cap B_{2i}) \mid \gamma < \kappa\}$ is locally finite in $U_{2i + 1} \setminus \cl{U_{2i - 2}}$,
 where $U_{-2} = \emptyset$.
Therefore
 $$\{g(A_\gamma \cap B_{2i}) \mid \gamma < \kappa, i \in \omega\}$$
 is locally finite in $U$.

Next, we can find an open cover $\mathcal{W}$ of $U$ that refines $\mathcal{U}'$ and satisfies the following:
\begin{itemize}
 \item For every map $h : A \to U$, if $h$ is $\mathcal{W}$-close to $g$,
 then the collection $\{h(A_\gamma \cap B_{2i}) \mid \gamma < \kappa, i \in \omega\}$ is locally finite in $U$.
\end{itemize}
Remark that $g(B_{2i + 1}) \subset U_{2i + 2} \setminus \cl{U_{2i - 1}}$ for each $i \in \omega$.
By the similar construction as $g$, we can take a map $h : A \to U$ so that $h$ is $\mathcal{W}$-close to $g$ and $\{h(A_\gamma \cap B_{2i + 1}) \mid \gamma < \kappa, i \in \omega\}$ is locally finite in $U$.
It follows from the definition of $\mathcal{W}$ that $\{h(A_\gamma \cap B_{2i}) \mid \gamma < \kappa, i \in \omega\}$ is locally finite in $U$.
Therefore $\{h(A_\gamma) \mid \gamma < \kappa\}$ is locally finite in $U$.
Moreover, $h$ is $\mathcal{U}$-close to $f$.
Thus the proof is complete.
\end{proof}

\begin{prop}\label{LFAP-homot.dense}
Let $\mathfrak{C}$ be a class and $X$ be an ANR with the $\kappa$-locally finite approximation property for $\mathfrak{C}$.
For any homotopy dense subset $Y \subset X$, $Y$ also has the $\kappa$-locally finite approximation property for $\mathfrak{C}$.
\end{prop}

\begin{proof}
Suppose that $f : A = \bigoplus_{\gamma < \kappa} A_\gamma \to Y$ is a map,
 where each $A_\gamma \in \mathfrak{C}$, and $\mathcal{U}$ is an open cover of $Y$.
For each $U \in \mathcal{U}$, there is an open set $\tilde{U}$ in $X$ such that $U = \tilde{U} \cap Y$.
Let $\tilde{\mathcal{U}} = \{\tilde{U} \mid U \in \mathcal{U}\}$ and $V = \bigcup \tilde{\mathcal{U}}$.
Take a star refinement $\mathcal{V}$ of $\tilde{\mathcal{U}}$ that covers $V$.
Since $V$ is an open subset of $X$ and $X$ has the $\kappa$-locally finite approximation property for $\mathfrak{C}$,
 $V$ also has the $\kappa$-locally finite approximation property for $\mathfrak{C}$ by Proposition~\ref{LFAP-op.}.
Hence we can find a map $g : A \to V$ so that $g$ is $\mathcal{V}$-close to $f$ and $\{g(A_\gamma) \mid \gamma < \kappa\}$ is locally finite in $V$.
Take an open cover $\mathcal{W}$ of $V$ that is a refinement of $\mathcal{V}$ and satisfies the following condition.
\begin{itemize}
 \item For any map $h : A \to V$ that is $\mathcal{W}$-close to $g$, the collection $\{h(A_\gamma) \mid \gamma < \kappa\}$ is locally finite in $V$.
\end{itemize}
Since $Y$ is homotopy dense in $X$,
 so is it in $V$,
 which implies that there exists a map $\phi : V \to V$ such that $\phi$ is $\mathcal{W}$-close to the identity map of $V$ and $\phi(V) \subset Y$.
Let $h = \phi g$,
 so it is $\mathcal{W}$-close to $g$.
Hence $h$ is $\mathcal{U}$-close to $f$ and $\{h(A_\gamma) \mid \gamma < \kappa\}$ is locally finite in $Y$.
The proof is complete.
\end{proof}

We denote the cardinality of a set $X$ by $\card{X}$.
The next proposition detects the locally finite approximation property for a class $\mathfrak{C}$ of manifolds modeled on spaces with it.

\begin{prop}\label{LFAP-g-hered.}
Let $\mathfrak{C}$ be a closed hereditary class.
An ANR $X$ has the $\kappa$-locally finite approximation property for $\mathfrak{C}$ if and only if each point $x \in X$ has an open neighborhood with the $\kappa$-locally finite approximation property for $\mathfrak{C}$.
\end{prop}

\begin{proof}
The ``only if'' part is obvious.
We shall prove the ``if'' part.
By virtue of Theorem~2.6.5 of \cite{Sa10}, we need only to show that the $\kappa$-locally finite approximation property for $\mathfrak{C}$ satisfies the following conditions.
\begin{enumerate}
\renewcommand{\labelenumi}{(\alph{enumi})}
 \item For any open subsets $V \subset U$ of $X$, if $U$ has the $\kappa$-locally finite approximation property for $\mathfrak{C}$,
 then so does $V$.
 \item For any open sets $U, V \subset X$, if both $U$ and $V$ have the $\kappa$-locally finite approximation property for $\mathfrak{C}$,
 then so does the union $U \cup V$.
 \item For every discrete family $\{U_\gamma \mid \gamma < \lambda\}$ of open sets in $X$, if each $U_\gamma$ has the $\kappa$-locally finite approximation property for $\mathfrak{C}$,
 then so does the union $\bigcup_{\gamma < \lambda} U_\gamma$.
\end{enumerate}
As is easily observed,
 (c) holds.
The condition (a) follows from Proposition~\ref{LFAP-op.}.

We shall show the condition (b).
For the simplicity, we write $W = U \cup V$.
Let $f : A = \bigoplus_{\gamma < \kappa} A_{\gamma} \to W$ be a map,
 where $A_\gamma \in \mathfrak{C}$ for all $\gamma < \kappa$,
 and $\mathcal{W}$ be an open cover of $W$.
We need only to construct a map $h : A \to W$ such that $h$ is $\mathcal{W}$-close to $f$ and the family $\{h(A_\gamma) \mid \gamma < \kappa\}$ is locally finite in $W$.
Take open sets $W_i \subset W$, $i = 1, 2, 3, 4$, so that
 $$U \setminus V \subset W_1 \subset \cl_W{W_1} \subset W_2 \subset \cl_W{W_2} \subset W_3 \subset \cl_W{W_3} \subset W_4 \subset \cl_W{W_4} \subset U,$$
 where $\cl_W{W_i}$ is the closure of $W_i$ in $W$.
Then we can find an open cover $\mathcal{W}'$ of $W$ such that $\mathcal{W}'$ is a star-refinement of $\mathcal{W}$ and refines
 $$\{W_1, W_2 \cap V, W_3 \setminus \cl_W{W_1}, W_4 \setminus \cl_W{W_2}, W \setminus \cl_W{W_3}\}.$$
Since $U$ has the $\kappa$-locally finite approximation property for $\mathfrak{C}$ and each $A_\gamma \cap f^{-1}(\cl_W{W_4}) \in \mathfrak{C}$,
 that is a closed hereditary class,
 there exists a map $g : f^{-1}(\cl_W{W_4}) \to U$ such that $g$ is $\mathcal{W}'|_U$-homotopic to $f|_{f^{-1}(\cl_W{W_4})}$ and $\{g(A_\gamma \cap f^{-1}(\cl_W{W_4})) \mid \gamma < \kappa\}$ is locally finite in $U$.
By the Homotopy Extension Theorem, we can obtain a map $\tilde{g} : A \to W$ so that $\tilde{g}$ is $\mathcal{W}'$-close to $f$ and $\tilde{g}|_{f^{-1}(\cl_W{W_4})} = g$.
Then $\{g(A_\gamma \cap f^{-1}(\cl_W{W_3})) \mid \gamma < \kappa\}$ is locally finite in $W$.
Indeed, fix any point $x \in W$.
When $x \in U$,
 there exists an open neighborhood $U_x$ of $x$ in $U$ such that
 $$\card\{\gamma < \kappa \mid \tilde{g}(A_\gamma \cap f^{-1}(\cl_{W}{W_3})) \cap U_x \neq \emptyset\} < \infty$$
 because $\{g(A_\gamma \cap f^{-1}(\cl_W{W_3})) \mid \gamma < \kappa\}$ is locally finite in $U$.
Remark that $U_x$ is open in $W$.
When $x \in V \setminus U$,
 the subset $V \setminus \cl_W{W_4}$ is an open neighborhood of $x$ in $W$ and
 $$\tilde{g}(A_\gamma \cap f^{-1}(\cl_W{W_3})) \cap V \setminus \cl_W{W_4} = \emptyset$$
 since $\tilde{g}(f^{-1}(\cl_W{W_3})) \subset W_4$.

Let $\mathcal{W}''$ be an open cover of $W$ that refines $\mathcal{W}'$ and
 $$\{W_2, W_3 \setminus \cl_W{W_1}, W_4 \setminus \cl_W{W_2}, W \setminus \cl_W{W_3}\},$$
 and satisfies the following:
 \begin{itemize}
  \item For every map $h : f^{-1}(\cl_W{W_3}) \to W$, if $h$ is $\mathcal{W}''$-close to $\tilde{g}|_{f^{-1}(\cl_W{W_3})}$,
 then $\{h(A_\gamma \cap f^{-1}(\cl_W{W_3}) \mid \gamma < \kappa\}$ is locally finite in $W$.
 \end{itemize}
Note that each $A_\gamma \cap f^{-1}(V \setminus W_1) \in \mathfrak{C}$ and $\tilde{g}(f^{-1}(V \setminus W_1)) \subset V$.
Using the $\kappa$-locally finite approximation property of $V$, we can find a map $h : A \to W$ such that $h$ is $\mathcal{W}''$-close to $\tilde{g}$ and $\{h(A_\gamma \cap f^{-1}(V \setminus W_3)) \mid \gamma < \kappa\}$ is locally finite in $W$ by the same argument as the above.
It is easy to see that $h$ is $\mathcal{W}$-close to $f$.
Due to the definition of $\mathcal{W}''$, $\{h(A_\gamma \cap f^{-1}(\cl_W{W_3})) \mid \gamma < \kappa\}$ is locally finite in $W$.
Therefore $\{h(A_\gamma) \mid \gamma < \kappa\}$ is locally finite in $W$.
The proof is complete.
\end{proof}

Replacing $\mathfrak{M}_0(n)$ with $\mathfrak{C}$ and $D_\gamma = \I^{2n + 1}$ with $D_\gamma = C$ in the proof of Proposition~4.3 of \cite{Kos1} respectively, we have the following lemma.

\begin{lem}\label{loc.fin.approx.}
Let $\mathfrak{C}$ be a class and $C$ be a space such that every space of $\mathfrak{C}$ can be embedded into $C$ as a closed subset.
Suppose that $W$ is an open subset of an ANR $X$ which is contractible in $X$.
If $X$ has the $\kappa$-locally finite approximation property for $\{C\}$,
 then $W$ has the $\kappa$-locally finite approximation property for $\mathfrak{C}$.
\end{lem}

The class of compact spaces of dimension $\leq n$ is denoted by $\mathfrak{M}_0^n$.
As a corollary of Proposition~\ref{LFAP-g-hered.}, we can establish the following:

\begin{cor}\label{DAP-g-hered.}
Let $X$ be an ANR.
Then the following are equivalent.
\begin{enumerate}
 \item $X$ has the $\kappa$-discrete $n$-cells property for any $n \in \omega$.
 \item Each point $x \in X$ has an open neighborhood with the $\kappa$-discrete $n$-cells property for any $n \in \omega$.
 \item Each point $x \in X$ has an open neighborhood with the $\kappa$-discrete approximation property for $\mathfrak{M}_0^n$ for any $n \in \omega$.
 \item $X$ has the $\kappa$-discrete approximation property for $\mathfrak{M}_0^n$ for any $n \in \omega$.
\end{enumerate}
\end{cor}

\begin{proof}
The implications (4) $\Rightarrow$ (1) $\Rightarrow$ (2) are obvious.
Combining Lemma~\ref{loc.fin.approx.} (cf.~\cite[Proposition~4.3]{Kos1}) with \cite[Lemma~4.2]{Kos1}, we have (2) $\Rightarrow$ (3).
The implication (3) $\Rightarrow$ (4) follows from Propositition~\ref{LFAP-g-hered.} and \cite[Lemma~4.2]{Kos1}.
\end{proof}

In Proposition~4.4 of \cite{Kos1}, it is shown that the $\kappa$-discrete approximation property for $\mathfrak{C} \subset \mathfrak{M}_0$ can be replaced with the stronger property.
Even if we omit some conditions,
 that proposition holds.
Applying similar techniques, we can prove the following lemma.
For the sake of completeness, we shall give the proof.

\begin{lem}\label{str.discr.approx.}
Let $\mathfrak{C}$ be a class of spaces and $X$ be an ANR with the $\kappa$-discrete approximation property for $\mathfrak{C}$.
The following holds:
\begin{itemize}
 \item Let $f : A = \bigoplus_{\gamma < \kappa}A_\gamma \to X$ be a map, where each $A_\gamma \in \mathfrak{C}$,
 $B$ be a closed subset of $A$, and $\mathcal{U}$ be an open cover of $X$.
 Suppose that each $\cl{f(A_\gamma \cap B)}$ is a strong $Z$-set in $X$ and the collection $\{f(A_\gamma \cap B) \mid \gamma < \kappa\}$ is discrete in $X$.
 Then there exists a map $g : A \to X$ such that $g$ is $\mathcal{U}$-close to $f$, $g|_B = f|_B$ and $\{g(A_\gamma) \mid \gamma < \kappa\}$ is discrete in $X$.
\end{itemize}
\end{lem}

\begin{proof}
Choose open covers $\mathcal{U}_1$ and $\mathcal{U}_2$ of $X$ so that $\mathcal{U}_1$ is a star-refinement of $\mathcal{U}$ and $\mathcal{U}_2$ is a refinement of $\mathcal{U}_1$.
Let $B_\gamma = A_\gamma \cap B$ for each $\gamma < \kappa$.
Since $\{f(B_\gamma) \mid \gamma < \kappa\}$ is discrete in $X$,
 we can obtain a discrete collection $\{U_\gamma \mid \gamma < \kappa\}$ of open sets in $X$ so that $\cl{f(B_\gamma)} \subset U_\gamma$ for each $\gamma < \kappa$.
It follows from Proposition~\ref{Z-op.union}~(2) that the discrete union $\cl{f(B)} = \bigoplus_{\gamma < \kappa} \cl{f(B_\gamma)}$ is a strong $Z$-set in $X$.
Taking an open cover $\mathcal{U}_2'$ of $X$ that refines $\mathcal{U}_2$ and
 $$\{U_{\gamma}, X \setminus \cl{f(B)} \mid \gamma < \kappa\},$$
 we can find a $\mathcal{U}_2'$-homotopy $h : X \times \I \to X$ and an open neighborhood $W$ of $\cl{f(B)}$ in $X$ so that $h(x,0) = x$ for any $x \in X$ and $h(X \times \{1\}) \subset X \setminus W$.
Define $h' : A \times \I \to X$ by $h'(x,t) = h(f(x),t)$,
 so $h'$ is a $\mathcal{U}_2'$-homotopy and $h'(x,0) = f(x)$ for any $x \in A$.
Note that $h'(B_\gamma \times \I) \subset U_\gamma$ for every $\gamma < \kappa$.
By the compactness of $\I$, each $B_\gamma$ has an open neighborhood $V_\gamma$ in $A_\gamma$ such that $h'(V_\gamma \times \I) \subset U_\gamma$.
Take a map $k : A \to \I$ such that $k^{-1}(0) = B$ and $k^{-1}(1) = A \setminus \bigcup_{\gamma < \kappa} V_\gamma$, and define the map $f' : A \to X$ by $f'(x) = h'(x,k(x))$.
Let $W_\gamma = W \cap U_\gamma$ for each $\gamma < \kappa$.
Observe that $f'$ is $\mathcal{U}_2'$-close to $f$, $f'|_B = f|_B$ and
\begin{itemize}
 \item[($\ast$)] $f'(A \setminus V_\gamma) \cap W_\gamma = \emptyset$ for any $\gamma < \kappa$.
\end{itemize}

Let $W'_\gamma$ and $W''_\gamma$ be open neighborhoods of $\cl{f(B_\gamma)}$ for each $\gamma < \kappa$ such that
 $$\cl{W''_\gamma} \subset W'_\gamma \subset \cl{W'_\gamma} \subset W_\gamma.$$
Note that both $\{\cl{W'_\gamma} \mid \gamma < \kappa\}$ and $\{\cl{W''_\gamma} \mid \gamma < \kappa\}$ are discrete families of closed sets in $X$.
Choose an open cover $\mathcal{U}_1' $ of $X$ such that $\mathcal{U}_1'$ refines $\mathcal{U}_1$ and 
 $$\Bigg\{W''_\gamma, W'_\gamma \setminus \cl{f(B_\gamma)}, W_\gamma \setminus \cl{W''_\gamma}, X \setminus \bigcup_{\gamma' < \kappa} \cl{W'_{\gamma'}} \ \Bigg| \ \gamma < \kappa\Bigg\}.$$
Using the $\kappa$-discrete approximation property for $\mathfrak{C}$, we can obtain a $\mathcal{U}_1'$-homotopy $h'' : A \times \I \to X$ so that $h''(x,0) = f'(x)$ for all $x \in A$ and
\begin{itemize}
 \item[($\ast\ast$)] $\{h''(A_\gamma \times \{1\}) \mid \gamma < \kappa\}$ is discrete in $X$.
\end{itemize}
Since $h''$ is a $\mathcal{U}_1'$-homotopy and $h''(x,0) = f(x)$ for any $x \in B$,
 we have that $h''(B_\gamma \times \I) \subset W''_\gamma$ for each $\gamma < \kappa$.
Then there exists an open neighborhood $G_\gamma$ of $B_\gamma$ in $A_\gamma$ for every $\gamma < \kappa$ such that $h''(G_\gamma \times \I) \subset W''_\gamma$.
Taking a map $k' : A \to \I$ such that $(k')^{-1}(0) = B$ and $(k')^{-1}(1) = A \setminus \bigcup_{\gamma < \kappa} G_\gamma$, we can obtain the desired map $g : A \to X$ defined by $g(x) = h''(x,k'(x))$.
Then $g$ is $\mathcal{U}'_1$-close to $f'$,
 which means that $g$ is $\mathcal{U}$-close to $f$,
 and $g|_B = f|_B$.
It remains to show that $\{g(A_\gamma) \mid \gamma < \kappa\}$ is discrete in $X$.

Fix any $x \in X$.
According to ($\ast\ast$), the collection $\{g(A_\gamma \setminus G_\gamma) \mid \gamma < \kappa\}$ is discrete in $X$.
Therefore the point $x$ has an open neighborhood $U_x$ such that
 $$\card{\{\gamma < \kappa \mid g(A_\gamma \setminus G_\gamma) \cap U_x \neq \emptyset\}} \leq 1.$$
In the case that $x \in X \setminus \bigcup_{\gamma < \kappa} \cl{W''_\gamma}$, the subset
 $$U_x' = U_x \setminus \bigcup_{\gamma < \kappa} \cl{W''_\gamma}$$
 is an open neighborhood of $x$ in $X$.
Since each $g(G_\gamma) \subset W''_\gamma$,
 we have $U_x' \cap g(G_\gamma) = \emptyset$,
 which implies that
 $$\card{\{\gamma < \kappa \mid g(A_\gamma) \cap U_x' \ne \emptyset\}} \leq 1.$$
In the case that $x \in \bigcup_{\gamma < \kappa} \cl{W''_\gamma}$,
 there exists the unique $\gamma_0 < \kappa$ such that $x \in \cl{W''_{\gamma_0}}$ .
Then $U_x' = U_x \setminus \bigcup_{\gamma \ne \gamma_0} \cl{W''_\gamma}$ is an open neighborhood of $x$ such that $U_x' \cap g(G_\gamma) = \emptyset$ for all $\gamma \neq \gamma_0$.
Remark that $x \notin \cl{g(A_\gamma \setminus G_\gamma)}$ for any $\gamma \neq \gamma_0$.
Suppose not,
 so we have $x \in \cl{g(A_\gamma \setminus G_\gamma)}$ for some $\gamma \neq \gamma_0$.
Then there is a sequence $\{a_n\} \subset A_\gamma \setminus G_\gamma$ such that $g(a_n) \to x$.
We may assume that $\{g(a_n)\} \subset W'_{\gamma_0}$.
Then $\{f'(a_n)\} \subset W_{\gamma_0}$ because $g$ is $\mathcal{U}_1'$-close to $f'$.
On the other hand, since $A_\gamma \subset A \setminus V_{\gamma_0}$,
 it follows from ($\ast$) that $f'(A_\gamma) \cap W_{\gamma_0} = \emptyset$.
This is a contradiction.
Taking the open neighborhood
 $$U_x'' = U_x' \setminus \bigcup_{\gamma \neq \gamma_0} \cl{g(A_\gamma \setminus G_\gamma)}$$
 of $x$, we have that
 $$\card{\{\gamma < \kappa \mid g(A_\gamma) \cap U_x'' \ne \emptyset\}} \leq 1.$$
Consequently, $\{g(A_\gamma) \mid \gamma < \kappa\}$ is discrete in $X$.
\end{proof}

For a subclass $\mathfrak{C} \subset \mathfrak{M}_0$, the $\kappa$-locally finite approximation property for $\mathfrak{C}$ is coincident with the $\kappa$-discrete approximation property for $\mathfrak{C}$,
 see Lemma~4.2 of \cite{Kos1} (cf.~Lemma~4.6 of \cite{Ba1}).
We prove the following:

\begin{lem}\label{LFAP-DAP}
Let $\mathfrak{C}$ be a class and $X$ be an ANR.
Suppose that for every map $\phi : A \to X$, where $A \in \mathfrak{C}$,
 $\cl{\phi(A)}$ is a $Z$-set in $X$.
If $X$ has the $\kappa$-locally finite approximation property for $\mathfrak{C}$,
 then $X$ has the $\kappa$-discrete approximation property for $\mathfrak{C}$.
\end{lem}

\begin{proof}
We will show that for each map $f : \bigoplus_{\gamma < \kappa} A_\gamma \to X$, where $A_\gamma \in \mathfrak{C}$,
 and each open cover $\mathcal{U}$ of $X$, there is a map $h : \bigoplus_{\gamma < \kappa} A_\gamma \to X$ such that $h$ is $\mathcal{U}$-close to $f$ and $\{h(A_\gamma) \mid \gamma < \kappa\}$ is discrete in $X$.
Take an open cover $\mathcal{V}$ of $X$ that is a star-refinement of $\mathcal{U}$.
Applying the $\kappa$-locally finite approximation property for $\mathfrak{C}$, we can obtain a map $g : \bigoplus_{\gamma < \kappa} A_\gamma \to X$ so that $g$ is $\mathcal{V}$-close to $f$ and $\{g(A_\gamma) \mid \gamma < \kappa\}$ is locally finite in $X$.
Then the family $\{\cl{g(A_\gamma)} \mid \gamma < \kappa\}$ is also locally finite.
Choose an open cover $\mathcal{W}$ of $X$ such that it refines $\mathcal{V}$ and the following condition is satisfied.
\begin{itemize}
 \item For any map $\phi : \bigoplus_{\gamma < \kappa} A_\gamma \to X$ that is $\mathcal{W}$-close to $g$, the family $\{\phi(A_\gamma) \mid \gamma < \kappa\}$ is locally finite in $X$.
\end{itemize}

By transfinite induction, for each $\gamma < \kappa$, we can find a map $h_\gamma : A_\gamma \to X$ such that $h_\gamma$ is $\mathcal{W}$-close to $g|_{A_\gamma}$ and $\{\cl{h_\gamma(A_\gamma)} \mid \gamma < \kappa\}$ is pairwise disjoint.
Then the map $h : \bigoplus_{\gamma < \kappa} A_\gamma \to X$ defined by $h|_{A_\gamma} = h_\gamma$ is the desired map.
Suppose that $\gamma < \kappa$ and $h_{\gamma'}$ have been obtained for any $\gamma' < \gamma$.
By the assumption, $\cl{h_{\gamma'}(A_{\gamma'})}$ is a $Z$-set in $X$.
The family $\{\cl{h_{\gamma'}(A_{\gamma'})} \mid \gamma' < \gamma\}$ is locally finite in $X$ by the definition of $\mathcal{W}$,
 and hence the locally finite union $\bigcup_{\gamma' < \gamma} \cl{h_{\gamma'}(A_{\gamma'})}$ is a $Z$-set due to Proposition~\ref{Z-op.union}~(2).
Therefore, taking a map $\phi : X \to X$ so that $\phi$ is $\mathcal{W}$-close to the identity map on $X$ and $\phi(X) \cap \bigcup_{\gamma' < \gamma} \cl{h_{\gamma'}(A_{\gamma'})} = \emptyset$, we can define $h_\gamma = \phi g|_{A_\gamma}$.
Thus the proof is complete.
\end{proof}

On the locally finite approximation property of product spaces, the following proposition holds.

\begin{prop}\label{DAP-prod.}
Let $\mathfrak{C}$ be a topological and closed hereditary class and $X$ be an ANR.
If $X$ has the $\kappa$-locally finite approximation property for $\mathfrak{C}$,
 then for any space $Y$, the product space $X \times Y$ also has the $\kappa$-locally finite approximation property for $\mathfrak{C}$.
\end{prop}

\begin{proof}
Taking admissible metrics $d_X$ and $d_Y$ on $X$ and $Y$ respectively, we define an admissible metric $d$ on $X \times Y$ by
 $$d((x,y),(x',y')) = d_X(x,x') + d_Y(y,y')$$
 for $x, x' \in X$ and $y, y' \in Y$.
We shall show that for every map $f : A = \bigoplus_{\gamma < \kappa} A_{\gamma} \to X \times Y$, where $A_\gamma \in \mathfrak{C}$ for each $\gamma < \kappa$,
 and for every map $\alpha : X \times Y \to (0,1]$, there exists a map $h : A \to X \times Y$ such that for each $a \in A$, $d(h(a),f(a)) < \alpha(f(a))$ and the family $\{h(A_{\gamma}) \mid \gamma < \kappa\}$ is locally finite in $X \times Y$.

Let $\pr_X : X \times Y \to X$ and $\pr_Y : X \times Y \to Y$ be the projections,
 and
 $$B_n = \{a \in A \mid \alpha(f(a)) \geq 2^{-n}\},$$
 $n \in \omega$.
By induction, we will construct maps $f_n : A \to X$, $n \in \omega$, so that
\begin{enumerate}
 \item $f_n|_{B_{n - 2} \cup (A \setminus B_{n + 1})} = f_{n - 1}|_{B_{n - 2} \cup (A \setminus B_{n + 1})}$,
 \item $d_X(f_n(a),f_{n - 1}(a)) < 2^{-n - 2}$ for every $a \in A$,
 \item $\{f_n(A_\gamma \cap B_n) \mid \gamma < \kappa\}$ is locally finite in $X$,
\end{enumerate}
 where $f_{-1} = \pr_Xf$ and $B_{-2} = B_{-1} = \emptyset$.
Suppose that $f_i$, $i \leq n - 1$, have been obtained.
Due to the inductive assumption, the family $\{f_{n - 1}(A_\gamma \cap B_{n - 1}) \mid \gamma < \kappa\}$ is locally finite in $X$.
Hence there exists an open cover $\mathcal{U}$ of $X$ such that the mesh $< 2^{-n - 2}$ and the following holds.
\begin{itemize}
 \item For every map $\phi : A \to X$, if $\phi$ is $\mathcal{U}$-close to $f_{n - 1}$,
 then $\{\phi(A_\gamma \cap B_{n - 1}) \mid \gamma < \kappa\}$ is locally finite in $X$.
\end{itemize}
Note that each $A_\gamma \cap B_n \setminus \intr{B_{n - 1}} \in \mathfrak{C}$,
 where $\intr{B_{n - 1}}$ is the interior of $B_{n - 1}$.
Applying the $\kappa$-locally finite approximation property for $\mathfrak{C}$ of $X$, we can take a map $\phi : B_n \setminus \intr{B_{n - 1}} \to X$ so that $\phi$ is $\mathcal{U}$-homotopic to $f_{n - 1}|_{B_n \setminus \intr{B_{n - 1}}}$ and
 $$\{\phi(A_\gamma \cap B_n \setminus \intr{B_{n - 1}}) \mid \gamma < \kappa\}$$
 is locally finite in $X$.
By the Homotopy Extension Theorem, there is a $\mathcal{U}$-close map $f_n : A \to X$ to $f_{n - 1}$,
 which implies that $d_X(f_n(a),f_{n - 1}(a)) < 2^{-n - 2}$ for any $a \in A$, such that $f_n|_{B_{n - 2} \cup (A \setminus B_{n + 1})} = f_{n - 1}|_{B_{n - 2} \cup (A \setminus B_{n + 1})}$ and $f_n|_{B_n \setminus \intr{B_{n - 1}}} = \phi$.
It follows from the definition of $\mathcal{U}$ that $\{f_n(A_\gamma \cap B_{n - 1}) \mid \gamma < \kappa\}$ is locally finite in $X$,
 and hence $\{f_n(A_\gamma \cap B_n) \mid \gamma < \kappa\}$ is locally finite in $X$.

After completing the inductive construction, we can define a map $g : A \to X$ by $g|_{B_n} = f_{n + 1}|_{B_n}$.
Then $d_X(g(a),\pr_X(f(a))) < \alpha(f(a))$ for each $a \in A$.
Indeed, fix any $a \in A$,
 so we can find $n \in \omega$ so that $a \in B_n \setminus B_{n - 1}$,
 that is, $2^{-n} \leq \alpha(f(a)) < 2^{-n + 1}$.
Observe that
\begin{align*}
 d_X(g(a),\pr_X(f(a))) &= d_X(f_{n + 1}(a),f_{n - 2}(a))\\
 &\leq d_X(f_{n + 1}(a),f_n(a)) + d_X(f_n(a),f_{n - 1}(a)) + d_X(f_{n - 1}(a),f_{n - 2}(a))\\
 &< 2^{-n - 3} + 2^{-n - 2} + 2^{-n - 1} < 2^{-n} \leq \alpha(f(a)).
\end{align*}

Now we can obtain the desired map $h : A \to X \times Y$ defined by $h(a) = (g(a),\pr_Y(f(a)))$.
For each $a \in A$,
 $$d(h(a),f(a)) = d_X(g(a),\pr_X(f(a))) + d_Y(\pr_Y(f(a)),\pr_Y(f(a))) < \alpha(f(a)).$$
It remains to show that $\{h(A_{\gamma}) \mid \gamma < \kappa\}$ is locally finite in $X \times Y$.
Suppose conversely,
 so we can choose $(x,y) \in X \times Y$ and $a_i \in A_{\gamma_i}$,
 where $\gamma_i \neq \gamma_j$ if $i \neq j$, so that $h(a_i)$ converges to $(x,y)$.
When $\liminf \alpha(f(a_i)) = 0$,
 replacing $\{a_i\}$ with a subsequence, we have $\lim \alpha(f(a_i)) = 0$.
Since $d(h(a_i),f(a_i)) < \alpha(f(a_i))$ and $h(a_i) \to (x,y)$,
 $f(a_i)$ converges to $(x,y)$.
Then $\alpha(f(a_i)) \to \alpha(x,y) > 0$,
 which is a contradiction.
When $\liminf \alpha(f(a_i)) > 0$,
 there is $n \in \omega$ such that $2^{-n} \leq \liminf \alpha(f(a_i)) < 2^{-n + 1}$.
Taking a subsequence, we may assume that $2^{-n - 1} < \alpha(f(a_i)) < 2^{-n + 1}$.
Then any $a_i \in B_{n + 1}$.
Since $h(a_i) \to (x,y)$,
 we have
 $$f_{n + 2}(a_i) = g(a_i) = \pr_X(h(a_i)) \to x,$$
 which contradicts to that $\{f_{n + 2}(A_\gamma \cap B_{n + 1}) \mid \gamma < \kappa\}$ is locally finite in $X$.
Therefore $\{h(A_{\gamma}) \mid \gamma < \kappa\}$ is locally finite in $X \times Y$.
\end{proof}

\section{Characterizing absorbing sets in Hilbert manifolds}

In this section, we shall prove Theorem~\ref{abs.char.}.
Theorem~5.1 of \cite{BeMo} and Theorem~3.7 of \cite{SaY} are rewritten as follows\footnote{Theorem~5.1 of \cite{BeMo} and Theorem~3.7 of \cite{SaY} are valid without the assumption that classes are additive and contain $\I^n \times \kappa$.
Indeed, we can modify the proofs of Proposition~2.3 and Theorem~3.1 of \cite{BeMo}, and Proposition~3.5 of \cite{SaY}.}:

\begin{prop}\label{homot.equiv.}
Let $\mathfrak{C}$ be a topological and closed hereditary class, and $Y$ be a $\mathfrak{C}$-absorbing set in an $\ell_2(\kappa)$-manifold.
Suppose that a space $X \in \mathfrak{C}_\sigma$ of density $\leq \kappa$ is an ANR, is strongly $\mathfrak{C}$-universal, is a strong $Z_\sigma$-set in itself, and has the $\kappa$-discrete $n$-cells property for every $n \in \omega$.
Then every fine homotopy equivalence $f : Y \to X$ is a near-homeomorphism.
\end{prop}

Corollary~5.6~(i) of \cite{BeMo} and Theorem~3.9~(1) of \cite{SaY} hold without some conditions on classes.
The following proposition is proven in \cite{Sa11}\footnote{In \cite{BeMo,SaY}, Proposition~\ref{abs.set-H.mfd.} is shown by the Triangulation Theorem of Hilbert manifolds under the assumption that $\mathfrak{C}$ is $\I$-stable.
On the other hand, using the Open Embedding Theorem of Hilbert manifolds, K.~Sakai prove it without the $\I$-stability in \cite{Sa11}.}.
For the sake of completeness, we give its proof.

\begin{prop}\label{abs.set-H.mfd.}
Suppose that $\mathfrak{C}$ is a topological and closed hereditary class, and $\Omega$ is a $\mathfrak{C}$-absorbing set in $\ell_2(\kappa)$.
Then any $\ell_2(\kappa)$-manifold $M$ contains a $\mathfrak{C}$-absorbing set.
\end{prop}

\begin{proof}
By the Open Embedding Theorem of Hilbert manifolds, we can regard $M$ as an open subset of $\ell_2(\kappa)$.
Let $Y = M \cap \Omega$,
 so it is a $\mathfrak{C}$-absorbing set in $M$.
Indeed, as is easily observed,
 $Y$ is homotopy dense in $M$ because $\Omega$ is homotopy dense in $\ell_2(\kappa)$ and $M$ is open in $\ell_2(\kappa)$.
Since $Y$ is open in $\Omega$,
 $Y$ is in $\mathfrak{C}_\sigma$, is strong $\mathfrak{C}$-universal, and is a strong $Z_\sigma$-set in itself by Propositions~\ref{str.univ.op.} and \ref{Z-op.union}~(1).
Hence $Y$ is a $\mathfrak{C}$-absorbing set in $M$.
\end{proof}

Now we prove Theorem~\ref{abs.char.}.
The implication (2) $\Rightarrow$ (1) is shown in \cite{Sa11}.

\begin{proof}[Proof of Theorem~\ref{abs.char.}]
First, we show the implication (1) $\Rightarrow$ (3).
Since $X$ is an $\Omega$-manifold,
 each point $x \in X$ has an open neighborhood $U_x$ that is homeomorphic to some open set in $\Omega$.
Observe that $U_x$ is an ANR, is strong $\mathfrak{C}$-universal, and is a strong $Z_\sigma$-set in itself due to Propositions~\ref{str.univ.op.} and \ref{Z-op.union}~(1).
Since $\Omega$ is homotopy dense in $\ell_2(\kappa)$,
 that has the $\kappa$-discrete $n$-cells property for any $n \in \omega$,
 $\Omega$ has the $\kappa$-locally finite $n$-cells property by Proposition~\ref{LFAP-homot.dense}.
Moreover, it follows from Proposition~\ref{LFAP-op.} and \cite[Lemma~4.2]{Kos1} (cf.~\cite[Lemma~4.6]{Ba1}) that each $U_x$ also has the $\kappa$-discrete $n$-cells property for every $n \in \omega$.
According to 6.2.10~(4) of \cite{Sa10}, Propositions~\ref{str.univ.g-hered.} and \ref{str.Z_sigma-g-hered.}, and Corollary~\ref{DAP-g-hered.}, $X$ satisfies the conditions (a), (b), (c), (d) of (3).

Next, we prove (3) $\Rightarrow$ (2).
Combining Proposition~\ref{abs.set-H.mfd.} with Theorem~3.6 of \cite{SaY} (cf.~Theorem~4.2 of \cite{BeMo}), we can obtain a $\mathfrak{C}$-absorbing set $Y$ in some $\ell_2(\kappa)$-manifold and a fine homotopy equivalence $f : Y \to X$.
By Proposition~\ref{homot.equiv.}, $f$ is a near-homeomorphism,
 so $X$ is homeomorphic to $Y$.

Finally, we show (2) $\Rightarrow$ (1).
Regarding $X$ as a $\mathfrak{C}$-absorbing set in some $\ell_2(\kappa)$-manifold $M$, we can see that each point $x \in X$ has an open neighborhood $U_x$ in $M$ that is homeomorphic to $\ell_2(\kappa)$.
Then $X \cap U_x$ is a $\mathfrak{C}$-absorbing set in $U_x$ by the same argument in the proof of Proposition~\ref{abs.set-H.mfd.}.
Note that Theorem~3.1 of \cite{BeMo} holds without additivity of classes and separability (cf.~\cite{SaY}),
 so each point $x \in X$ has the open neighborhood $X \cap U_x$ homeomorphic to $\Omega$.
Therefore $X$ is an $\Omega$-manifold.
\end{proof}

\section{A $\bigoplus_{\kappa} \mathfrak{C}(\kappa')$-absorbing set in $\ell_2(\kappa)$}\label{abs.}

In this section, we shall show Theorem~\ref{abs.discr.}.
From now on, let $\aleph_0 \leq \kappa' < \kappa$, $\mathfrak{C}$ be a topological, closed hereditary and $\I$-stable class, and $\Omega$ be a $\mathfrak{C}(\kappa')$-absorbing set in $\ell_2(\kappa')$.

\begin{prop}\label{homot.dense}
The space $\ell_2^f(\kappa) \times \Omega$ is homotopy dense in $\ell_2(\kappa) \times \ell_2(\kappa')$.
\end{prop}

\begin{proof}
Since $\ell_2^f(\kappa)$ and $\Omega$ are homotopy dense in $\ell_2(\kappa)$ and $\ell_2(\kappa')$ respectively,
 $\ell_2^f(\kappa) \times \Omega$ is also homotopy dense in $\ell_2(\kappa) \times \ell_2(\kappa')$.
\end{proof}

\begin{prop}\label{class}
The space $\ell_2^f(\kappa) \times \Omega$ is in $(\bigoplus_{\kappa} \mathfrak{C}(\kappa'))_\sigma$.
\end{prop}

\begin{proof}
Since $\ell_2^f(\kappa)$ is strongly countable-dimensional, $\sigma$-locally compact and of density $\kappa$,
 and $\Omega \in (\mathfrak{C}(\kappa'))_\sigma$,
 we can write $\ell_2^f(\kappa) = \bigcup_{n \in \omega} (\bigoplus_{\gamma < \kappa} A_{(n,\gamma)})$ and $\Omega = \bigcup_{m \in \omega} \Omega_m$,
 where $A_{(n,\gamma)} \in \mathfrak{M}_0^{fd}$ and $\Omega_m \in \mathfrak{C}(\kappa')$ is closed in $\Omega$.
Then
 $$\ell_2^f(\kappa) \times \Omega = \bigcup_{n,m \in \omega} \Bigg(\bigoplus_{\gamma < \kappa} A_{(n,\gamma)} \times \Omega_m\Bigg).$$
Since $\mathfrak{C}$ is topological, closed hereditary and $\I$-stable,
 each closed subset $A_{(n,\gamma)} \times \Omega_m \in \mathfrak{C}(\kappa')$.
Hence $\ell_2^f(\kappa) \times \Omega \in (\bigoplus_{\kappa} \mathfrak{C}(\kappa'))_\sigma$.
\end{proof}

Let $X = (X,d)$ be a metric space.
For a class $\mathfrak{M}$ and a cardinal $\lambda$, we say that $X$ has \textit{the $\lambda$-discrete separation property for $\mathfrak{M}$} if the following is satisfied.
\begin{itemize}
 \item For each $\epsilon > 0$, there exists $\delta > 0$ such that for any map $f : \bigoplus_{\gamma < \lambda} A_\gamma \to X$, where each $A_\gamma \in \mathfrak{M}$,
 there is a map $g : \bigoplus_{\gamma < \lambda} A_\gamma \to X$ such that $g$ is $\epsilon$-homotopic to $f$ and $d(g(A_\gamma),g(A_{\gamma'})) \geq \delta$ for any $\gamma < \gamma' < \lambda$.
\end{itemize}
In the case that $\mathfrak{M} = \{\I^n\}$, $n \in \omega$, refer to \cite{BZ}.
By the same argument, Lemmas~5.2 and 6.1 of \cite{BZ} are generalized as follows:

\begin{lem}\label{DSP-DAP}
Let $\mathfrak{M}$ be a class.
If a metric space $X$ has the $\kappa$-discrete separation property for $\mathfrak{M}(\kappa')$,
 then $X$ has the $\kappa$-discrete approximation property for $\mathfrak{M}(\kappa')$.
\end{lem}

\begin{lem}\label{DSP-lin.}
Let $\mathfrak{M}$ be a class.
Suppose that a metric space $X = (X,\0)$ with the base point $\0$ has a pseudo-translation\footnote{Refer to Section~6 of \cite{BZ}.} with respect to $\0$,
 and is locally path-connected at $\0$.
Then $X$ has the $\kappa$-discrete separation property for $\mathfrak{M}(\kappa')$ if and only if any neighborhood of $\0$ contains a separated subset\footnotemark[6] of cardinality $\kappa$.
\end{lem}

Using these lemmas, we can obtain the following:

\begin{prop}\label{DAP-l2f}
Let $\mathfrak{M}$ be a class.
The space $\ell_2^f(\kappa)$ has the $\kappa$-discrete approximation property for $\mathfrak{M}(\kappa')$.
\end{prop}

\begin{proof}
Take an admissible linear metric $d$ on $\ell_2^f(\kappa)$ defined as follows:
 $$d(x,y) = \Bigg(\sum_{\gamma < \kappa} |x(\gamma) - y(\gamma)|^2\Bigg)^\frac{1}{2}$$
 for $x = (x(\gamma)), y = (y(\gamma)) \in \ell_2^f(\kappa)$.
By virtue of Lemmas~\ref{DSP-DAP} and \ref{DSP-lin.}, we only need to show that any neighborhood of the origin $\0 \in \ell_2^f(\kappa)$ contains a separated subset of cardinality $\kappa$.
Indeed, for each $\epsilon > 0$, the $\epsilon$-ball centered at $\0$ contains an $\epsilon$-discrete subset $\{(\epsilon/\sqrt{2})\e_\gamma \mid \gamma < \kappa\}$ of cardinality $\kappa$,
 where $\e_\gamma = (x(\gamma'))_{\gamma' < \kappa}$ is the unit element such that $x(\gamma') = 1$ if $\gamma' = \gamma$,
 and $x(\gamma') = 0$ if $\gamma' \neq \gamma$.
\end{proof}

Remark that $\mathfrak{M}_0^{fd} \subset \mathfrak{C}$ because $\mathfrak{C}$ is topological, closed hereditary and $\I$-stable.
Therefore any space with the $\kappa$-locally finite approximation property for $\mathfrak{C}$ has the $\kappa$-locally finite $n$-cells property for every $n \in \omega$,
 and hence it has the $\kappa$-discrete $n$-cells property by Lemma~4.2 of \cite{Kos1}.
Due to the combination of Propositions~\ref{DAP-prod.}, \ref{DAP-l2f} and Lemma~\ref{LFAP-DAP}, we have the following.

\begin{prop}\label{DAP}
The space $\ell_2^f(\kappa) \times \Omega$ has the $\kappa$-discrete approximation property for $\mathfrak{C}(\kappa')$.
\end{prop}

\begin{proof}
According to Proposition~\ref{DAP-l2f}, the space $\ell_2^f(\kappa)$ has the $\kappa$-discrete approximation property for $\mathfrak{C}(\kappa')$.
The product $\ell_2^f(\kappa) \times \Omega$ has the $\kappa$-locally finite approximation property for $\mathfrak{C}(\kappa')$ by Proposition~\ref{DAP-prod.},
 and hence it has the $\kappa$-discrete $n$-cells property for every $n \in \omega$ due to \cite[Lemma~4.2]{Kos1} (cf.~\cite[Lemma~4.6]{Ba1}).
Remark that for any map $\phi : A \to \ell_2^f(\kappa) \times \Omega$, where $A \in \mathfrak{C}(\kappa')$,
 the closure $\cl{\phi(A)}$ is of density $\leq \kappa' < \kappa$,
 so it is a $Z$-set in $\ell_2^f(\kappa) \times \Omega$ by Proposition~\ref{Z-DCP}~(1).
It follows from Lemma~\ref{LFAP-DAP} that $\ell_2^f(\kappa) \times \Omega$ has the $\kappa$-discrete approximation property for $\mathfrak{C}(\kappa')$.
\end{proof}

\begin{remark}
Proposition~\ref{DAP} can be proven by using Proposition~\ref{LFAP-homot.dense} instead of Proposition~\ref{DAP-prod.}.
Indeed, due to the same argument as Proposition~\ref{DAP-l2f}, the product $\ell_2(\kappa) \times \ell_2(\kappa')$, that is homeomorphic to $\ell_2(\kappa)$,
 has the $\kappa$-discrete approximation property for $\mathfrak{C}(\kappa')$.
Then the homotopy dense subset $\ell_2^f(\kappa) \times \Omega \subset \ell_2(\kappa) \times \ell_2(\kappa')$ has the $\kappa$-locally finite approximation property for $\mathfrak{C}(\kappa')$.
It follows from Lemma~\ref{LFAP-DAP} that $\ell_2^f(\kappa) \times \Omega$ has the $\kappa$-discrete approximation property for $\mathfrak{C}(\kappa')$.
\end{remark}

\begin{prop}\label{str.Z_sigma}
The space $\ell_2^f(\kappa) \times \Omega$ is a strong $Z_\sigma$-set in itself.
\end{prop}

\begin{proof}
Since $\ell_2^f(\kappa)$ (or $\Omega$) is a strong $Z_\sigma$-set in itself,
 so is $\ell_2^f(\kappa) \times \Omega$ due to Proposition~\ref{Z-prod.}.
\end{proof}

By virtue of Proposition~\ref{str.univ.prod.}, we have the following:

\begin{prop}\label{str.univ.}
The space $\ell_2^f(\kappa) \times \Omega$ is strongly $\mathfrak{C}(\kappa')$-universal.
\end{prop}

\begin{proof}
Since $\mathfrak{M}_0^{fd} \subset \mathfrak{C}$,
 the space $\ell_2^f(\kappa) \times \Omega$ has the $\kappa$-discrete $n$-cells property for any $n \in \omega$ by Proposition~\ref{DAP}.
Moreover, $\ell_2^f(\kappa) \times \Omega$ is a strong $Z_\sigma$-set in itself due to Proposition~\ref{str.Z_sigma}.
It follows from Proposition~\ref{Z-DCP}~(2) that every $Z$-set in $\ell_2^f(\kappa) \times \Omega$ is a strong $Z$-set.
Because $\Omega$ is strongly $\mathfrak{C}(\kappa')$-universal,
 the product space $\ell_2^f(\kappa) \times \Omega$ is also strongly $\mathfrak{C}(\kappa')$-universal by Proposition~\ref{str.univ.prod.}.
\end{proof}

\begin{proof}[Proof of Theorem~\ref{abs.discr.}]
According to Propositions~\ref{str.univ.}, \ref{DAP} and Lemma~\ref{str.discr.approx.}, the space $\ell_2^f(\kappa) \times \Omega$ is strongly $\bigoplus_{\kappa} \mathfrak{C}(\kappa')$-universal.
Combining this with Propositions~\ref{homot.dense}, \ref{class}, and \ref{str.Z_sigma}, we can conclude that $\ell_2^f(\kappa) \times \Omega$ is a $\bigoplus_{\kappa} \mathfrak{C}(\kappa')$-absorbing set in $\ell_2(\kappa) \times \ell_2(\kappa')$.
\end{proof}

\section{Characterizing $(\ell_2^f(\kappa) \times \Omega)$-manifolds}

This section is devoted to proving Main Theorem.
We establish the following theorem.

\begin{thm}
For a connected space $X \in (\bigoplus_{\kappa} \mathfrak{C}(\kappa'))_\sigma$, the following conditions are equivalent.
\begin{enumerate}
 \item $X$ is an $(\ell_2^f(\kappa) \times \Omega)$-manifold.
 \item
 \begin{enumerate}
  \item $X$ is an ANR.
  \item $X$ is strongly $\bigoplus_\kappa \mathfrak{C}(\kappa')$-universal.
  \item $X$ is a strong $Z_\sigma$-set in itself.
 \end{enumerate}
 \item
 \begin{enumerate}
  \item $X$ is an ANR.
  \item
  \begin{enumerate}
   \item $X$ is strongly $\mathfrak{C}(\kappa')$-universal.
   \item $X$ has the $\kappa$-discrete approximation property for $\mathfrak{C}(\kappa')$.
  \end{enumerate}
  \item $X$ is a strong $Z_\sigma$-set in itself.
 \end{enumerate}
 \item
 \begin{enumerate}
  \item $X$ is an ANR.
  \item
  \begin{enumerate}
   \item $X$ is strongly $\mathfrak{C}(\kappa')$-universal.
   \item $X$ has the $\kappa$-discrete approximation property for $\{\Omega\}$.
  \end{enumerate}
  \item $X$ is a strong $Z_\sigma$-set in itself.
 \end{enumerate}
 \item
 \begin{enumerate}
  \item $X$ is an ANR.
  \item
  \begin{enumerate}
   \item $X$ is strongly $\mathfrak{C}(\kappa')$-universal.
   \item $X$ has the $\kappa$-discrete $n$-cells property for every $n \in \omega$.
  \end{enumerate}
  \item $X$ is a strong $Z_\sigma$-set in itself.
 \end{enumerate}
\end{enumerate}
\end{thm}

\begin{proof}
Note that the strong $\bigoplus_\kappa \mathfrak{C}(\kappa')$-universality implies the $\kappa$-discrete $n$-cells property for any $n \in \omega$.
Due to Theorems~\ref{abs.char.} and \ref{abs.discr.}, the equivalence (1) $\Leftrightarrow$ (2) holds.
The implications (2) $\Rightarrow$ (3) $\Rightarrow$ (5) are obvious.
According to Proposition~\ref{str.univ.sigma}, (2) $\Rightarrow$ (4) holds because for each topological copy $A_\gamma$ of $\Omega \in (\mathfrak{C}(\kappa'))_\sigma$, $\gamma < \kappa$, the sum
 $$\bigoplus_{\gamma < \kappa} A_\gamma \in \bigoplus_\kappa (\mathfrak{C}(\kappa'))_\sigma \subset \Bigg(\bigoplus_\kappa \mathfrak{C}(\kappa')\Bigg)_\sigma.$$
Combining Lemma~\ref{str.discr.approx.} with the condition (b) of (3), we get (b) of (2).
Therefore (3) $\Rightarrow$ (2) holds.

We will prove the implication (4) $\Rightarrow$ (1).
Since $X$ is locally contractible,
 each point of $X$ has an open neighborhood $W$ which is contractible in $X$.
It is enough to show that $W$ is an $(\ell_2^f(\kappa) \times \Omega)$-manifold, that is, $W$ satisfies (3).
As is easily observed, the open set $W$ is an ANR and is in $(\bigoplus_{\kappa} \mathfrak{C}(\kappa'))_\sigma$.
Since $X$ is strongly $\mathfrak{C}(\kappa')$-universal and is a strong $Z_\sigma$-set in itself,
 so is $W$ by Propositions~\ref{str.univ.op.} and \ref{Z-op.union}~(1).
It follows from Lemma~\ref{loc.fin.approx.} that $W$ has the $\kappa$-locally finite approximation property for $\mathfrak{C}(\kappa')$.
Since $\kappa > \kappa'$, 
 the subset $W$ has the $\kappa$-discrete approximation property for $\mathfrak{C}(\kappa')$ by the combination of Proposition~\ref{Z-DCP}~(1) and Lemma~\ref{LFAP-DAP}.
Consequently, $W$ satisfies the condition (3).

It remains to show (5) $\Rightarrow$ (4).
Since $X$ has the $\kappa$-discrete $n$-cells property for every $n \in \omega$,
 $X$ has the $\kappa$-discrete approximation property for $\mathfrak{M}_0^n$ by Corollary~\ref{DAP-g-hered.}.
Note that $\mathfrak{M}_0^n \subset \mathfrak{C}$.
Hence $X$ is strongly $\bigcup_{n \in \omega} (\bigoplus_\kappa \mathfrak{M}_0^n)$-universal by virtue of Lemma~\ref{str.discr.approx.},
 which means that it is strongly $(\bigcup_{n \in \omega} (\bigoplus_\kappa \mathfrak{M}_0^n))_\sigma$-universal due to Proposition~\ref{str.univ.sigma}.
By the combination of Theorems~4.9.6 and 6.6.2 of \cite{Sa10}, the ANR $\Omega$ is approximated by an metric polyhedron\footnote{See Section~4.5 of \cite{Sa10}.}.
Remark that the polyhedron is in $(\bigcup_{n \in \omega} (\bigoplus_{\kappa'} \mathfrak{M}_0^n))_\sigma$ by Proposition~6.1 of \cite{Kos1}.
Therefore $\bigoplus_{\gamma < \kappa} A_\gamma$, where each $A_\gamma$ is homeomorphic to $\Omega$,
 can be approximated by some space belonging in $(\bigcup_{n \in \omega} (\bigoplus_\kappa \mathfrak{M}_0^n))_\sigma$,
 and hence $X$ has the $\kappa$-discrete approximation property for $\{\Omega\}$.
Thus the proof is complete.
\end{proof}

\section*{Acknowledgments}

The author would like to thank Professor Katsuro Sakai for his useful comments and suggestions that helped him to improve the paper.

\end{document}